\documentclass[11pt,a4paper]{article}
\usepackage{amsmath, amsfonts, amsthm, amssymb, color}
\usepackage{enumerate, hyperref}
\usepackage{graphicx}
\usepackage{tikz}

\usepackage{xcolor}
\def\red{\textcolor{red}}

\usetikzlibrary{intersections,calc,arrows.meta}

\theoremstyle{definition}
\newtheorem{thm}{Theorem}

\newtheorem{prop}[thm]{Proposition}

\newtheorem{prob}[thm]{Problem}

\begin{document}

\title{Spanning plane subgraphs of $1$-plane graphs}

\author{
Kenta Noguchi\thanks{Department of Information Sciences, Tokyo University of Science, 2641 Yamazaki, Noda, Chiba 278-8510, Japan.
E-mail address: {\tt noguchi@rs.tus.ac.jp}},
Katsuhiro Ota\thanks{Department of Mathematics, Keio University, 3-14-1 Hiyoshi, Kohoku-ku, Yokohama 223-8522, Japan.
E-mail address: {\tt ohta@math.keio.ac.jp}},
and Yusuke Suzuki\thanks{Department of Mathematics, Niigata University, 8050 Ikarashi 2-no-cho, Nishi-ku, Niigata 950-2181, Japan. 
E-mail address: {\tt y-suzuki@math.sc.niigata-u.ac.jp}}
}

\date{}
\maketitle

\noindent
\begin{abstract}

A graph drawn on the plane is called $1$-plane if each edge is crossed at most once by another edge. 
In this paper, we show that every $4$-connected $1$-plane graph has a connected spanning plane subgraph. 
We also show that there exist infinitely many $4$-connected $1$-plane graphs that have no $2$-connected spanning plane subgraphs. 
Moreover, we consider the condition of $k$ and $l$ such that every $k$-connected $1$-plane graph has an $l$-connected spanning plane subgraph.
\end{abstract}

\noindent
\textbf{Keywords.}
$1$-planar graph, 
$1$-plane drawing, 
spanning plane subgraph, 
connectivity

\section{Introduction}
\label{introsec}
For a given graph $G$, the problem to find a (spanning) subgraph of $G$ with certain properties is natural and practical, and it receives much attention in the literature; for example, bipartite subgraphs, regular subgraphs, highly connected subgraphs, and so on. 
Among others, we focus on finding planar subgraphs from a given graph. 
In the past, the thickness of a graph $G$, which is the least number $k$ such that the edge set $E(G)$ can be partitioned into $k$ planar graphs, are widely studied; see \cite{MOS} for a survey. 
In 2020, Fabrici et al.~\cite{FHMMSZ} studied the existence of planar subgraphs of $1$-planar graphs. 
Here a graph is \emph{$1$-planar} if it has a drawing on the plane so that each edge is crossed at most once by another edge. 
This class of graphs is well studied; see \cite{FHMMSZ,HMS,PT,Su}. 
One of their main theorems states that every $3$-connected locally maximal $1$-planar graph has a $3$-connected planar spanning subgraph, where a $1$-planar graph $G$ is locally maximal if there exists a $1$-plane drawing such that, for every pair of two crossing edges, their end vertices induce the complete graph $K_4$ \cite[Theorem 5]{FHMMSZ}. 
In the present paper, we consider its analog for \emph{plane} spanning subgraphs of graphs already drawn on the plane. 
The first and the third author have already dealt with one of such problems in \cite{NS}, namely, proved that every optimal $1$-plane graph $G$ has a $4$-connected spanning triangulation, where a $1$-planar graph $G$ of order $n (\ge 3)$ is optimal if $|E(G)| = 4n-8$; it is shown in \cite{PT} that every $1$-planar graph of order $n (\ge 3)$ has at most $4n-8$ edges. 
As a consequence, $G$ is hamiltonian by well-known Tutte's theorem~\cite{Tu1} stating that every $4$-connected planar graph is hamiltonian. 
This gave another proof of hamiltonicity of optimal $1$-planar graphs, which have been originally shown by Hud\'{a}k, Madaras, and the third author~\cite{HMS}. 

The main problem of this paper is, for a given $1$-plane graph $G$, to find a highly connected spanning subgraph of $G$ which is a plane graph with respect to the drawing of $G$. 
We call it a \emph{spanning plane subgraph} hereafter. 
It is natural to expect that the high connectivity of $G$ guarantees the existence of a highly connected spanning plane subgraph. 
Here the \emph{connectivity} $\kappa(G)$ and the \emph{edge-connectivity} $\kappa'(G)$ of a (multi)graph $G$ are defined by: 
\begin{eqnarray*}
\kappa(G) &:=& \min\{p(u, v) \mid u, v \in V(G), u\ne v\} ~\mbox{and} \\
\kappa'(G) &:=& \min\{p'(u, v) \mid u, v \in V(G), u\ne v\},
\end{eqnarray*}
respectively, where $p(u, v)$ is the maximum number of pairwise internally disjoint (vertex-disjoint) $uv$-paths and $p'(u, v)$ is the maximum number of pairwise edge-disjoint $uv$-paths. 
By the maximum size $4n-8$ of a 1-planar graph $G$, $G$ has a vertex of degree at most $7$ and $G$ cannot be $8$-connected. 
We show the following theorem. 

\begin{thm} 
\label{thm1}
Every $4$-connected $1$-plane graph has a connected spanning plane subgraph. 
\end{thm}

On the other hand, we also show the following theorems.

\begin{thm} 
\label{thm2}
There exist infinitely many $3$-connected $1$-plane graphs that have no connected spanning plane subgraphs. 
\end{thm}

\begin{thm} 
\label{thm3}
There exist infinitely many $4$-connected $1$-plane graphs 
that have no $2$-connected spanning plane subgraphs. 
\end{thm}

\begin{thm} 
\label{thm4}
There exist infinitely many $6$-connected $1$-plane graphs that have no $3$-connected spanning plane subgraphs. 
\end{thm}

This paper is organized as follows. 
In Section \ref{sec:def}, we define some notations and prepare useful tools. 
In Section \ref{sec:proof}, we show the main theorems. 
In the last section, we consider the sharpness of the connectivity on theorems above and we conclude remarks.
Throughout the paper, the \emph{minimum degree} and the \emph{maximum degree} of a (multi)graph $G$ are denoted by $\delta(G)$ and $\Delta(G)$, respectively.

\section{Drawing of a graph}
\label{sec:def}
For terms not defined here, see the standard textbook, e.g., \cite{BM}. 
We consider finite undirected graphs. 
Otherwise noted, a graph is assumed to be simple, while \emph{multigraph} can have parallel edges (but no loops). 
Let $G$ be a multigraph. 
A plane \emph{drawing} of $G$ is an image of a continuous map from $G$ to the plane; the vertices are distinct points and the edges are non-self-intersecting arcs, containing no vertex points in the interior, which join two points corresponding to their end vertices. 
A drawing is an \emph{embedding} if its map is injective, that is, there is no \emph{crossing point}. 
A multigraph $G$ is \emph{planar} if $G$ has a plane embedding. 
A graph $G$ is \emph{$1$-planar} if it has a plane drawing such that each edge has at most one crossing point. 
We often consider that a given $1$-planar graph $G$ is already drawn on the plane as such, and we denote its image by $G$ itself to simplify the notation. 
In this case, we say that $G$ is a \emph{$1$-plane graph}.
In the drawing of a graph $G$, the edge set can be divided into \emph{crossing edges} and \emph{non-crossing edges}. 
Any $1$-plane drawing forbids crossings of adjacent edges. 

An embedding of a planar multigraph (or a graph embedded on orientable surfaces in general) $G$ is represented by the so-called \emph{rotation system}, which is the set of the \emph{rotation} of the verices $v_i$ where the rotation of $v_i$ is a cyclic ordering of the incident edges given by the clockwise one appearing near by $v_i$, see \cite{GT, MT} for the detail. 
For a plane multigraph $G$, two adjacent edges $e_1 = vu$ and $e_2 =vw$ are \emph{consecutive} if $e_1$ and $e_2$ appear consecutively at the rotation of a vertex $v$. 
Note that a $1$-plane drawing of a $1$-planar graph $G$ can also be represented by a rotation system if we regard crossing points as (dummy) vertices, see \cite{Su} for example. 

We define three operations which are used in our proof of Theorems 2--4. 
The first operation, namely, the inflation, was originally defined by Fleischner and Jackson~\cite{FJ}. 

\medskip
\noindent
\textbf{Inflation and attaching gadgets}

Let $G$ be a multigraph with $V(G) = \{v_1, v_2, \ldots, v_n\}$ with $\delta(G)\geq 3$. 
We denote the edge joining $v_i$ and $v_j$ of $G$ by $e_{ij}$; note that $e_{ij} = e_{ji}$ (when $G$ has parallel edges, we consider the edge set as a multiset). 
Furthermore, we suppose that each vertex $v_i$ has a cyclic ordering of the incident edges of $v_i$. 
We denote the given cyclic orderings to all the vertices of $G$ by $\rho$, and denote the composite structure 
$G$ with $\rho$ by $G_{\rho}$. 
We define a new graph $(G_{\rho})_{\textrm{infl.}}$, so-called the \emph{inflation} of $G_{\rho}$ with 
\[
V((G_{\rho})_{\textrm{infl.}}) = \bigcup_{e_{ij} \in E(G)} \{v_{ij}, v_{ji} \} ~\mbox{and}
\]
\[
E((G_{\rho})_{\textrm{infl.}}) = \{v_{ij}v_{ji} \mid e_{ij} \in E(G) \} \cup 
\{v_{ip}v_{iq} \mid e_{ip}, e_{iq} \in E(G)~\textrm{are consecutive around $v_i$ in}~\rho \}. 
\]
The inflation $(G_{\rho})_{\textrm{infl.}}$ is simple and cubic. 
We also refer the operation above as an inflation applied to $G_{\rho}$. 
In contrast to the notation of edges, note that $v_{ij}$ and $v_{ji}$ are distinct vertices in $(G_{\rho})_{\textrm{infl.}}$, which are intuitively regarded as ``two ends of $e_{ij} \in E(G)$''. 
Note that in $(G_{\rho})_{\textrm{infl.}}$, each vertex $v$ of $G$ is replaced with a cycle $C_v$ of length $\deg(v)$, 
where the corresponding end vertices of incident edges to $v$ appear on $C_v$ in the order with respect to $\rho$. 
The $C_v$ above is called an {\it inflated cycle\/} of $v$. 
Conversely, when we shrink each inflated cycle of $(G_{\rho})_{\textrm{infl.}}$ to a vertex, we naturally obtain $G$ (or $G_{\rho}$). 
Note that edges joining different inflated cycles 1-to-1 correspond to the edges of $G$. 

In our latter arguments, we usually discuss a graph $G$ already drawn (or embedded) on the plane. 
As we mentioned, each vertex $v$ of $G$ has a cyclic ordering, which is of the rotation of $v$.  
In this case, we consider that ``$G$'' means the composite structure mentioned above (without the subscript $\rho$), 
and we simply denote the inflation obtained from $G$ by $G_{\textrm{infl.}}$. 
We sometimes refer $G_{\textrm{infl.}}$ obtained from $G$ drawn on the plane as a \emph{canonical} inflation. 
The drawing of the canonical inflation $G_{\textrm{infl.}}$ is naturally obtained from $G$, that is, each vertex $v$ of $G$ is replaced with a small inflated cycle $C_v$ around the location of $v$, and edges corresponding to those incident to $v$ is placed radially from $C_v$ keeping its ordering. 
In this case, if we shrink each inflated cycle of $G_{\textrm{infl.}}$ to a vertex, we naturally obtain $G$ drawn on the plane. 

Let $G$ be a graph with $\delta(G) \ge 3$ and $\Delta(G) \le 7$, and $(G_{\rho})_{\textrm{infl.}}$ be an inflation of $G_{\rho}$ for some $\rho$. 
We define a graph $(G_{\rho})_{\textrm{infl.}}^{+}$ obtained from $(G_{\rho})_{\textrm{infl.}}$ 
by attaching, to each inflated cycle except for triangles, the graph shown in Figure \ref{gadget}, which we call \emph{$k$-gadgets}. 
For example, for an inflated cycle $C$ of length $5$, we identify $C$ and the \emph{outer} $5$-cycle of the (copy of) second graph in Figure~\ref{gadget} facing the outer region. 
Similar to the notation above, if $G$ is a graph drawn on the plane, then we simply use the notation like $G_{\textrm{infl.}}^{+}$

Now we prepare the following propositions; the first one is straightforward by the arguments above. 

\begin{prop}
\label{prop_1-planarity}
If $G$ is a plane multigraph (resp. $1$-plane graph) with $\delta(G) \ge 3$, then the canonical inflation $G_{\textrm{infl.}}$ has a plane embedding (resp. $1$-plane drawing). 
If $G$ is a $1$-plane graph with $\delta(G) \ge 3$ and $\Delta(G) \le 7$, 
then $G_{\textrm{infl.}}^{+}$ is also a $1$-plane graph. 
\end{prop}

\begin{prop}
\label{prop_gadget}
Fix $k\in \{4, 5, 6, 7\}$ and let $G_k$ be a $k$-gadget. 
Let $C = v_1\cdots v_k$ be the outer $k$-cycle of $G_k$. 
Let $G'_k$ be the graph obtained from $G_k$ by adding a new vertex $x$ and edges $xv_i$ for all $i \in \{1, \ldots, k\}$. 
Then, the graph $G'_k$ is $k$-regular and $k$-connected. 
\end{prop}

\begin{proof}
By construction, the $k$-regularity is trivial. 
It is tedious but straightforward to find $k$ vertex-disjoint $uv$-paths in $G'_k$ for every pair of distinct vertices $u, v \in V(G'_k)$.
\end{proof}

\begin{prop}
\label{prop_k-conn}
Fix $k\in \{3, 4, 5, 6, 7\}$ and let $G$ be a $k$-edge connected $k$-regular graph. 
Then, for any $\rho$ given to $G$, the graph $(G_{\rho})_{\textrm{infl.}}^{+}$ is $k$-regular and $k$-connected. 
\end{prop}

\begin{proof}
Again, the $k$-regularity is trivial. 
For every pair of distinct vertices $u, v \in V((G_{\rho})_{\textrm{infl.}}^{+})$, one can take $k$ vertex-disjoint $uv$-paths in $(G_{\rho})_{\textrm{infl.}}^{+}$ as follows. If $u$ and $v$ belong to the same gadgets, then, by Proposition \ref{prop_gadget}, there are $k-1$ vertex-disjoint $uv$-paths in the gadget and one such path possibly through other gadgets. 
If $u$ and $v$ belong to different gadgets, then there are $k$ edge-disjoint paths, between two vertices that correspond to the two gadgets, in $G$ since $G$ is $k$-edge connected. Hence, it is not difficult to see that these $k$ paths correspond to $k$ vertex-disjoint $uv$-paths in $(G_{\rho})_{\textrm{infl.}}^{+}$. 
(Note that we can take such $k$-vertex disjoint paths using edges corresponding to those of $k$-edge disjoint paths in $G$, and using suitable paths in gadgets.)
\end{proof}

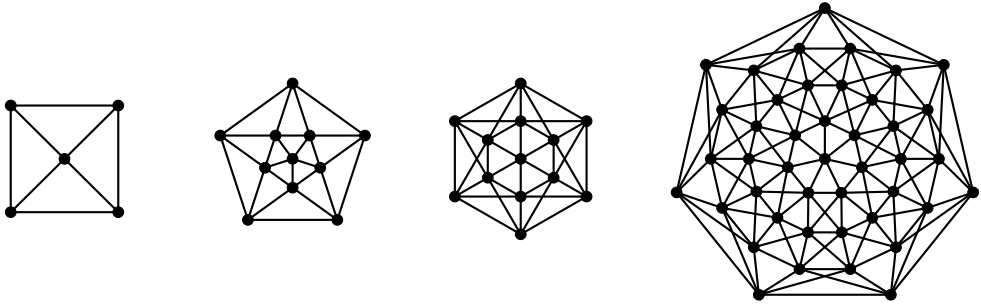
\begin{figure}[th]
\centering
\begin{tikzpicture}
\coordinate (A0) at ({-10+0},0);
\coordinate (A) at ({-10+cos(pi/4 r)},{sin(pi/4 r)});
\coordinate (B) at ({-10+cos(3*pi/4 r)},{sin(3*pi/4 r)});
\coordinate (C) at ({-10+cos(5*pi/4 r)},{sin(5*pi/4 r)});
\coordinate (D) at ({-10+cos(7*pi/4 r)},{sin(7*pi/4 r)});

\coordinate (B0) at ({-7+0},0);
\coordinate (E) at ({-7+0},1);
\coordinate (F) at ({-7+cos(9*pi/10 r)},{sin(9*pi/10 r)});
\coordinate (G) at ({-7+cos(13*pi/10 r)},{sin(13*pi/10 r)});
\coordinate (H) at ({-7+cos(17*pi/10 r)},{sin(17*pi/10 r)});
\coordinate (I) at ({-7+cos(21*pi/10 r)},{sin(21*pi/10 r)});

\coordinate (C0) at ({-4+0},0);
\coordinate (J) at ({-4+0},1);
\coordinate (K) at ({-4+cos(5*pi/6 r)},{sin(5*pi/6 r)});
\coordinate (L) at ({-4+cos(7*pi/6 r)},{sin(7*pi/6 r)});
\coordinate (M) at ({-4+0},{-1});
\coordinate (N) at ({-4+cos(11*pi/6 r)},{sin(11*pi/6 r)});
\coordinate (O) at ({-4+cos(13*pi/6 r)},{sin(13*pi/6 r)});

\draw[fill=black] (A0) circle (2pt);
\draw[fill=black] (A) circle (2pt);
\draw[fill=black] (B) circle (2pt);
\draw[fill=black] (C) circle (2pt);
\draw[fill=black] (D) circle (2pt);

\draw[fill=black] (B0) circle (2pt);
\draw[fill=black] (E) circle (2pt);
\draw[fill=black] (F) circle (2pt);
\draw[fill=black] (G) circle (2pt);
\draw[fill=black] (H) circle (2pt);
\draw[fill=black] (I) circle (2pt);

\draw[fill=black] (C0) circle (2pt);
\draw[fill=black] (J) circle (2pt);
\draw[fill=black] (K) circle (2pt);
\draw[fill=black] (L) circle (2pt);
\draw[fill=black] (M) circle (2pt);
\draw[fill=black] (N) circle (2pt);
\draw[fill=black] (O) circle (2pt);

\draw[thick] (A)--(B)--(C)--(D)--(A);
\draw[thick] (A)--(C);
\draw[thick] (B)--(D);

\draw[thick] (E)--(F)--(G)--(H)--(I)--(E);
\draw[name path=b1,thick] (E)--(G);
\draw[name path=b2,thick] (G)--(I);
\draw[name path=b3,thick] (I)--(F);
\draw[name path=b4,thick] (F)--(H);
\draw[name path=b5,thick] (H)--(E);
\path[name intersections={of=b1 and b3,by=b6}];
\path[name intersections={of=b2 and b4,by=b7}];
\path[name intersections={of=b3 and b5,by=b8}];
\path[name intersections={of=b4 and b1,by=b9}];
\path[name intersections={of=b5 and b2,by=b10}];
\draw[fill=black] (b6) circle (2pt);
\draw[fill=black] (b7) circle (2pt);
\draw[fill=black] (b8) circle (2pt);
\draw[fill=black] (b9) circle (2pt);
\draw[fill=black] (b10) circle (2pt);
\draw[thick] (B0)--(b6);
\draw[thick] (B0)--(b7);
\draw[thick] (B0)--(b8);
\draw[thick] (B0)--(b9);
\draw[thick] (B0)--(b10);

\draw[thick] (J)--(K)--(L)--(M)--(N)--(O)--(J);
\draw[name path=c1,thick] (J)--(L);
\draw[name path=c2,thick] (K)--(M);
\draw[name path=c3,thick] (L)--(N);
\draw[name path=c4,thick] (M)--(O);
\draw[name path=c5,thick] (N)--(J);
\draw[name path=c6,thick] (O)--(K);
\draw[name path=ca,thick] (J)--(M);
\draw[name path=cb,thick] (K)--(N);
\draw[name path=cc,thick] (L)--(O);

\path[name intersections={of=c1 and cb,by=c7}];
\path[name intersections={of=c2 and cc,by=c8}];
\path[name intersections={of=c3 and ca,by=c9}];
\path[name intersections={of=c4 and cb,by=c10}];
\path[name intersections={of=c5 and cc,by=c11}];
\path[name intersections={of=c6 and ca,by=c12}];
\draw[fill=black] (c7) circle (2pt);
\draw[fill=black] (c8) circle (2pt);
\draw[fill=black] (c9) circle (2pt);
\draw[fill=black] (c10) circle (2pt);
\draw[fill=black] (c11) circle (2pt);
\draw[fill=black] (c12) circle (2pt);
\draw[thick] (c7)--(c8)--(c9)--(c10)--(c11)--(c12)--(c7);

\coordinate (A_0) at (0,0);
\foreach \k in{1,...,7}
    \coordinate (A_\k) at ({0.5*cos((4*\k+3)*pi/14 r)},{0.5*sin((4*\k+3)*pi/14 r)});
\foreach \k in{1,...,14}
    \coordinate (B_\k) at ({cos(\k*pi/7 r)},{sin(\k*pi/7 r)});
\foreach \k in{1,...,14}
    \coordinate (C_\k) at ({1.5*cos(\k*pi/7 r)},{1.5*sin(\k*pi/7 r)});
\foreach \k in{1,...,7}
    \coordinate (D_\k) at ({2*cos((4*\k+3)*pi/14 r)},{2*sin((4*\k+3)*pi/14 r)});

\foreach \k in{0,...,7}
   \draw[fill=black] (A_\k) circle (2pt);
\foreach \k in{1,...,14}
   \draw[fill=black] (B_\k) circle (2pt);
\foreach \k in{1,...,14}
   \draw[fill=black] (C_\k) circle (2pt);
\foreach \k in{1,...,7}
   \draw[fill=black] (D_\k) circle (2pt);

\draw[thick] (A_1)--(A_2)--(A_3)--(A_4)--(A_5)--(A_6)--(A_7)--(A_1);
\draw[thick] (B_1)--(B_2)--(B_3)--(B_4)--(B_5)--(B_6)--(B_7)--(B_8)--(B_9)--(B_10)--(B_11)--(B_12)--(B_13)--(B_14)--(B_1);
\draw[thick] (C_1)--(C_2)--(C_3)--(C_4)--(C_5)--(C_6)--(C_7)--(C_8)--(C_9)--(C_10)--(C_11)--(C_12)--(C_13)--(C_14)--(C_1);
\draw[thick] (D_1)--(D_2)--(D_3)--(D_4)--(D_5)--(D_6)--(D_7)--(D_1);

\foreach \k in{1,...,7}
   \draw[thick] (A_0)--(A_\k);
\foreach \k in{1,...,14}
   \draw[thick] (B_\k)--(C_\k);

\draw[thick] (A_1)--(B_3)--(C_4)--(D_1)--(C_3)--(B_4)--(A_1)--(B_5)--(C_4)--(D_2)--(C_5)--(B_6)--(A_2)--(B_5)--(C_6)--(D_2)--(C_7)--(B_6)--(A_3);
\draw[thick] (A_3)--(B_7)--(C_8)--(D_3)--(C_7)--(B_8)--(A_3)--(B_9)--(C_8)--(D_4)--(C_9)--(B_10)--(A_4)--(B_9)--(C_10)--(D_4)--(C_11)--(B_10)--(A_5);
\draw[thick] (A_5)--(B_11)--(C_12)--(D_5)--(C_11)--(B_12)--(A_5)--(B_13)--(C_12)--(D_6)--(C_13)--(B_14)--(A_6)--(B_13)--(C_14)--(D_6)--(C_1)--(B_14)--(A_7);
\draw[thick] (A_7)--(B_1)--(C_2)--(D_7)--(C_1)--(B_2)--(A_7)--(B_3)--(C_2)--(D_1)--(C_5)--(B_4)--(A_2)--(B_7)--(C_6)--(D_3)--(C_9)--(B_8)--(A_4);
\draw[thick] (A_4)--(B_11)--(C_10)--(D_5)--(C_13)--(B_12)--(A_6)--(B_1)--(C_14)--(D_7)--(C_3)--(B_2)--(A_1);
\end{tikzpicture}
 \caption{4-, 5-, 6-, and 7-gadgets.}
 \label{gadget}
\end{figure}

\medskip
\noindent
\textbf{Crossing operation}

Let $G$ be a plane multigraph with $\delta(G) \ge 3$ and $\Delta(G) \le 7$. 
Suppose that $G$ has an even cycle $C = v_1e_1v_2e_2 \cdots v_{2k}e_{2k}v_1$ ($k=1$ is allowed\red{;} $e_1$ and $e_2$ consist of parallel edges). 
Suppose further that two edges $e_{2j-1}$ and $e_{2j}$ are consecutive in the rotation of $v_{2j}$ for $j \in \{1, \ldots, k \}$; 
note that this condition always occurs if $G$ is cubic. 
Recall that $G_{\textrm{infl.}}^{+}$ denotes a $1$-plane graph obtained from $G$ by applying the canonical inflation and attaching a gadget to each inflated cycle. 
In $G_{\textrm{infl.}}^{+}$, we still refer to the $2k$ edges corresponding to those of $C$ as $e_1, \ldots, e_{2k}$. 
Then, one can get a new $1$-plane graph $H$ from $G_{\textrm{infl.}}^{+}$ by crossing $e_{2j-1}$ and $e_{2j}$ for each $j \in \{1, \ldots, k \}$; 
that is, we exchange two end vertices of $e_{2j-1}$ and $e_{2j}$ of $G_{\textrm{infl.}}^{+}$, which are vertices in the inflated cycle of $v_{2j}$; see Figure~\ref{fig:2} for example. 
We say that $H$ is obtained from $G_{\textrm{infl.}}^{+}$ by a {\it crossing operation\/} along the even cycle $C$ 
of (the original graph) $G$. 
We can define it not only for an even cycle but also for a path $P = v_1e_1v_2e_2 \cdots v_{2k}e_{2k}v_{2k+1}$ of even length $2k$ for a positive integer $k$; in that case, we similarly cross edges $e_1$ and $e_2$, $e_3$ and $e_4$, and so on.
(For the operation on an even cycle, if the first vertex is not suggested, then take arbitrary one from two choices.) 
Furthermore, we consider a crossing operation along some disjoint cycles and paths of $G$ simultaneously. 
Notice that $H$ is a graph isomorphic to $(G_{\rho})_{\textrm{infl.}}^{+}$ for some $\rho$; $\rho$ probably does not coincide with the rotation system of $G$, that is, $H$ is obtained from a noncanonical inflation of $G$ by attaching gadgets. 

\begin{figure}[th]
\centering
\begin{tikzpicture}[scale=0.7]
\def\y{11}
\def\z{10}
\coordinate [label=right: {$G$}] (fig:1) at (-4,3.5);
\coordinate [label=right: {$G_{\textrm{infl.}}$}] (fig:2) at (\y-4,3.5);
\coordinate [label=right: {$G_{\textrm{infl.}}^{+}$}] (fig:3) at (-4,-6.5);
\coordinate [label=right: {$(G_{\rho})_{\textrm{infl.}}^{+}$}] (fig:4) at (\y-4,-6.5);
\coordinate [label=center: {$e_1$}] (1_1) at ({3*cos(pi/3 r)},{3*sin(pi/3 r)});
\coordinate [label=center: {$e_6$}] (1_2) at ({3*cos(2*pi/3 r)},{3*sin(2*pi/3 r)});
\coordinate [label=center: {$e_5$}] (1_3) at (-3,0);
\coordinate [label=center: {$e_4$}] (1_4) at ({3*cos(4*pi/3 r)},{3*sin(4*pi/3 r)});
\coordinate [label=center: {$e_3$}] (1_5) at ({3*cos(5*pi/3 r)},{3*sin(5*pi/3 r)});
\coordinate [label=center: {$e_2$}] (1_6) at (3,0);

\draw [->, >=Latex] (4.5,0) -- (6.5,0);
\draw [->, >=Latex] (6.5,-4) -- (4.5,-6);
\draw [->, >=Latex] (4.5,-10) -- (6.5,-10);

\coordinate [label=above right: {$v_1$}] (A_1) at (0,{3*1});
\coordinate [label=above: {$v_2$}] (A_2) at ({3*cos(pi/6 r)},{3*sin(pi/6 r)});
\coordinate [label=above right: {$v_3$}] (A_3) at ({3*cos(-pi/6 r)},{3*sin(-pi/6 r)});
\coordinate [label=above: {$v_4$}] (A_4) at (0,{3*-1});
\coordinate [label=above left: {$v_5$}] (A_5) at ({3*cos(7*pi/6 r)},{3*sin(7*pi/6 r)});
\coordinate [label=above left: {$v_6$}] (A_6) at ({3*cos(5*pi/6 r)},{3*sin(5*pi/6 r)});
\foreach \k in{1,...,6}
   \draw[fill=black] (A_\k) circle (2pt);

\draw[thick] (A_1)--(A_2)--(A_3)--(A_4)--(A_5)--(A_6)--(A_1);
\draw[thick] (A_1)--(0,4);
\draw[thick] (A_2)--(30:4);
\draw[thick] (A_3)--({1+3*cos(-pi/6 r)},{3*sin(-pi/6 r)});
\draw[thick] (330:2)--(330:4);
\draw[thick] (A_4)--(0,-4);
\draw[thick] (A_4)--(0.5,{-(6+sqrt(3))/2});
\draw[thick] (A_4)--(-0.5,{-(6+sqrt(3))/2});
\draw[thick] (210:2)--(210:4);
\draw[thick] (A_6)--(150:2);
\draw[thick] (A_6)--({1+3*cos(5*pi/6 r)},{3*sin(5*pi/6 r)});

\coordinate (B_1) at ({\y+0},{3*1});
\coordinate (B_2) at ({\y+3*cos(pi/6 r)},{3*sin(pi/6 r)});
\coordinate (B_3) at ({\y+3*cos(-pi/6 r)},{3*sin(-pi/6 r)});
\coordinate (B_4) at ({\y+0},{3*-1});
\coordinate (B_5) at ({\y+3*cos(7*pi/6 r)},{3*sin(7*pi/6 r)});
\coordinate (B_6) at ({\y+3*cos(5*pi/6 r)},{3*sin(5*pi/6 r)});

\foreach \k in{1,...,3}
   {
   \coordinate (B1_\k) at ({\y+0.5*cos((4*\k-1)*pi/6 r)},{3*1+0.5*sin((4*\k-1)*pi/6 r});
   \draw[fill=black] (B1_\k) circle (2pt);
   \coordinate (B2_\k) at ({\y+3*cos(pi/6 r)+0.5*cos((4*\k-3)*pi/6 r)},{3*sin(pi/6 r)+0.5*sin((4*\k-3)*pi/6 r});
   \draw[fill=black] (B2_\k) circle (2pt);
   }
\foreach \k in{1,...,5}
   {
   \coordinate (B3_\k) at ({\y+3*cos(-pi/6 r)+0.5*cos((4*\k+1)*pi/10 r)},{3*sin(-pi/6 r)+0.5*sin((4*\k+1)*pi/10 r});
   \draw[fill=black] (B3_\k) circle (2pt);
   \coordinate (B4_\k) at ({\y+0.5*cos((4*\k-1)*pi/10 r)},{3*-1+0.5*sin((4*\k-1)*pi/10 r});
   \draw[fill=black] (B4_\k) circle (2pt);
   }
\foreach \k in{1,...,4}
   {
   \coordinate (B5_\k) at ({\y+3*cos(7*pi/6 r)+0.5*cos((2*\k-1)*pi/4 r)},{3*sin(7*pi/6 r)+0.5*sin((2*\k-1)*pi/4 r});
   \draw[fill=black] (B5_\k) circle (2pt);
   \coordinate (B6_\k) at ({\y+3*cos(5*pi/6 r)+0.5*cos((2*\k-1)*pi/4 r)},{3*sin(5*pi/6 r)+0.5*sin((2*\k-1)*pi/4 r});
   \draw[fill=black] (B6_\k) circle (2pt);
   }

\draw[thick] (B1_1)--(B1_2)--(B1_3)--(B1_1);
\draw[thick] (B2_1)--(B2_2)--(B2_3)--(B2_1);
\draw[thick] (B3_1)--(B3_2)--(B3_3)--(B3_4)--(B3_5)--(B3_1);
\draw[thick] (B4_1)--(B4_2)--(B4_3)--(B4_4)--(B4_5)--(B4_1);
\draw[thick] (B5_1)--(B5_2)--(B5_3)--(B5_4)--(B5_1);
\draw[thick] (B6_1)--(B6_2)--(B6_3)--(B6_4)--(B6_1);
\draw[thick] (B1_1)--++(0,1);
\draw[thick] (B1_2)--(B6_2);
\draw[thick] (B1_3)--(B2_2);
\draw[thick] (B2_1)--++({sqrt(3)/2},0.5);
\draw[thick] (B2_3)--(B3_1);
\draw[thick] (B3_2)--++(-1,0);
\draw[thick] (B3_3)--(B4_1);
\draw[thick] (B3_4)--++({sqrt(3)/2},-0.5);
\draw[thick] (B3_5)--++(1,0); 
\draw[thick] (B4_2)--(B5_4);
\draw[thick] (B4_3)--++({-sqrt(3)/2},-0.5);
\draw[thick] (B4_4)--++(0,-1);
\draw[thick] (B4_5)--++({sqrt(3)/2},-0.5);
\draw[thick] (B5_1)--++({sqrt(2)/2},{sqrt(2)/2});
\draw[thick] (B5_2)--(B6_3);
\draw[thick] (B5_3)--++({-sqrt(2)/2},{-sqrt(2)/2});
\draw[thick] (B6_1)--++(1,0);
\draw[thick] (B6_4)--++({sqrt(3)/2},-0.5);

\coordinate (C_1) at (0,{-\z+3*1});
\coordinate (C_2) at ({3*cos(pi/6 r)},{-\z+3*sin(pi/6 r)});
\coordinate (C_3) at ({3*cos(-pi/6 r)},{-\z+3*sin(-pi/6 r)});
\coordinate (C_4) at (0,{-\z+3*-1});
\coordinate (C_5) at ({3*cos(7*pi/6 r)},{-\z+3*sin(7*pi/6 r)});
\coordinate (C_6) at ({3*cos(5*pi/6 r)},{-\z+3*sin(5*pi/6 r)});
\draw[fill=black] (C_3) circle (2pt);
\draw[fill=black] (C_4) circle (2pt);
\draw[fill=black] (C_5) circle (2pt);
\draw[fill=black] (C_6) circle (2pt);

\foreach \k in{1,...,3}
   {
   \coordinate (C1_\k) at ({0.5*cos((4*\k-1)*pi/6 r)},{-\z+3*1+0.5*sin((4*\k-1)*pi/6 r});
   \draw[fill=black] (C1_\k) circle (2pt);
   \coordinate (C2_\k) at ({3*cos(pi/6 r)+0.5*cos((4*\k-3)*pi/6 r)},{-\z+3*sin(pi/6 r)+0.5*sin((4*\k-3)*pi/6 r});
   \draw[fill=black] (C2_\k) circle (2pt);
   }

\foreach \k in{1,...,5}
   {
   \coordinate (C3_\k) at ({3*cos(-pi/6 r)+0.5*cos((4*\k+1)*pi/10 r)},{-\z+3*sin(-pi/6 r)+0.5*sin((4*\k+1)*pi/10 r});
   \draw[fill=black] (C3_\k) circle (2pt);
   \coordinate (C4_\k) at ({0.5*cos((4*\k-1)*pi/10 r)},{-\z+3*-1+0.5*sin((4*\k-1)*pi/10 r});
   \draw[fill=black] (C4_\k) circle (2pt);
   }

\foreach \k in{1,...,4}
   {
   \coordinate (C5_\k) at ({3*cos(7*pi/6 r)+0.5*cos((2*\k-1)*pi/4 r)},{-\z+3*sin(7*pi/6 r)+0.5*sin((2*\k-1)*pi/4 r});
   \draw[fill=black] (C5_\k) circle (2pt);
   \coordinate (C6_\k) at ({3*cos(5*pi/6 r)+0.5*cos((2*\k-1)*pi/4 r)},{-\z+3*sin(5*pi/6 r)+0.5*sin((2*\k-1)*pi/4 r});
   \draw[fill=black] (C6_\k) circle (2pt);
   }

\draw[thick] (C1_1)--(C1_2)--(C1_3)--(C1_1);
\draw[thick] (C2_1)--(C2_2)--(C2_3)--(C2_1);
\draw[thick] (C3_1)--(C3_2)--(C3_3)--(C3_4)--(C3_5)--(C3_1);
\draw[thick] (C4_1)--(C4_2)--(C4_3)--(C4_4)--(C4_5)--(C4_1);
\draw[thick] (C5_1)--(C5_2)--(C5_3)--(C5_4)--(C5_1)--(C5_3)--(C5_2)--(C5_4);
\draw[thick] (C6_1)--(C6_2)--(C6_3)--(C6_4)--(C6_1)--(C6_3)--(C6_2)--(C6_4);

\draw[name path=c1,thick] (C3_1)--(C3_3);
\draw[name path=c2,thick] (C3_2)--(C3_4);
\draw[name path=c3,thick] (C3_3)--(C3_5);
\draw[name path=c4,thick] (C3_4)--(C3_1);
\draw[name path=c5,thick] (C3_5)--(C3_2);
\path[name intersections={of=c1 and c2,by=c6}];
\path[name intersections={of=c2 and c3,by=c7}];
\path[name intersections={of=c3 and c4,by=c8}];
\path[name intersections={of=c4 and c5,by=c9}];
\path[name intersections={of=c5 and c1,by=c10}];
\draw[fill=black] (c6) circle (2pt);
\draw[fill=black] (c7) circle (2pt);
\draw[fill=black] (c8) circle (2pt);
\draw[fill=black] (c9) circle (2pt);
\draw[fill=black] (c10) circle (2pt);
\draw[thick] (C_3)--(c6);
\draw[thick] (C_3)--(c7);
\draw[thick] (C_3)--(c8);
\draw[thick] (C_3)--(c9);
\draw[thick] (C_3)--(c10);

\draw[name path=c11,thick] (C4_1)--(C4_3);
\draw[name path=c12,thick] (C4_2)--(C4_4);
\draw[name path=c13,thick] (C4_3)--(C4_5);
\draw[name path=c14,thick] (C4_4)--(C4_1);
\draw[name path=c15,thick] (C4_5)--(C4_2);
\path[name intersections={of=c11 and c12,by=c16}];
\path[name intersections={of=c12 and c13,by=c17}];
\path[name intersections={of=c13 and c14,by=c18}];
\path[name intersections={of=c14 and c15,by=c19}];
\path[name intersections={of=c15 and c11,by=c20}];
\draw[fill=black] (c16) circle (2pt);
\draw[fill=black] (c17) circle (2pt);
\draw[fill=black] (c18) circle (2pt);
\draw[fill=black] (c19) circle (2pt);
\draw[fill=black] (c20) circle (2pt);
\draw[thick] (C_4)--(c16);
\draw[thick] (C_4)--(c17);
\draw[thick] (C_4)--(c18);
\draw[thick] (C_4)--(c19);
\draw[thick] (C_4)--(c20);

\draw[thick] (C1_1)--++(0,1);
\draw[thick] (C1_2)--(C6_2);
\draw[thick] (C1_3)--(C2_2);
\draw[thick] (C2_1)--++({sqrt(3)/2},0.5);
\draw[thick] (C2_3)--(C3_1);
\draw[thick] (C3_2)--++(-1,0);
\draw[thick] (C3_3)--(C4_1);
\draw[thick] (C3_4)--++({sqrt(3)/2},-0.5);
\draw[thick] (C3_5)--++(1,0); 
\draw[thick] (C4_2)--(C5_4);
\draw[thick] (C4_3)--++({-sqrt(3)/2},-0.5);
\draw[thick] (C4_4)--++(0,-1);
\draw[thick] (C4_5)--++({sqrt(3)/2},-0.5);
\draw[thick] (C5_1)--++({sqrt(2)/2},{sqrt(2)/2});
\draw[thick] (C5_2)--(C6_3);
\draw[thick] (C5_3)--++({-sqrt(2)/2},{-sqrt(2)/2});
\draw[thick] (C6_1)--++(1,0);
\draw[thick] (C6_4)--++({sqrt(3)/2},-0.5);

\coordinate (D_1) at ({\y+0},{-\z+3*1});
\coordinate (D_2) at ({\y+3*cos(pi/6 r)},{-\z+3*sin(pi/6 r)});
\coordinate (D_3) at ({\y+3*cos(-pi/6 r)},{-\z+3*sin(-pi/6 r)});
\coordinate (D_4) at ({\y+0},{-\z+3*-1});
\coordinate (D_5) at ({\y+3*cos(7*pi/6 r)},{-\z+3*sin(7*pi/6 r)});
\coordinate (D_6) at ({\y+3*cos(5*pi/6 r)},{-\z+3*sin(5*pi/6 r)});
\draw[fill=black] (D_3) circle (2pt);
\draw[fill=black] (D_4) circle (2pt);
\draw[fill=black] (D_5) circle (2pt);
\draw[fill=black] (D_6) circle (2pt);

\foreach \k in{1,...,3}
   {
   \coordinate (D1_\k) at ({\y+0.5*cos((4*\k-1)*pi/6 r)},{-\z+3*1+0.5*sin((4*\k-1)*pi/6 r});
   \draw[fill=black] (D1_\k) circle (2pt);
   \coordinate (D2_\k) at ({\y+3*cos(pi/6 r)+0.5*cos((4*\k-3)*pi/6 r)},{-\z+3*sin(pi/6 r)+0.5*sin((4*\k-3)*pi/6 r});
   \draw[fill=black] (D2_\k) circle (2pt);
   }

\foreach \k in{1,...,5}
   {
   \coordinate (D3_\k) at ({\y+3*cos(-pi/6 r)+0.5*cos((4*\k+1)*pi/10 r)},{-\z+3*sin(-pi/6 r)+0.5*sin((4*\k+1)*pi/10 r});
   \draw[fill=black] (D3_\k) circle (2pt);
   \coordinate (D4_\k) at ({\y+0.5*cos((4*\k-1)*pi/10 r)},{-\z+3*-1+0.5*sin((4*\k-1)*pi/10 r});
   \draw[fill=black] (D4_\k) circle (2pt);
   }

\foreach \k in{1,...,4}
   {
   \coordinate (D5_\k) at ({\y+3*cos(7*pi/6 r)+0.5*cos((2*\k-1)*pi/4 r)},{-\z+3*sin(7*pi/6 r)+0.5*sin((2*\k-1)*pi/4 r});
   \draw[fill=black] (D5_\k) circle (2pt);
   \coordinate (D6_\k) at ({\y+3*cos(5*pi/6 r)+0.5*cos((2*\k-1)*pi/4 r)},{-\z+3*sin(5*pi/6 r)+0.5*sin((2*\k-1)*pi/4 r});
   \draw[fill=black] (D6_\k) circle (2pt);
   }

\draw[thick] (D1_1)--(D1_2)--(D1_3)--(D1_1);
\draw[thick] (D2_1)--(D2_2)--(D2_3)--(D2_1);
\draw[thick] (D3_1)--(D3_2)--(D3_3)--(D3_4)--(D3_5)--(D3_1);
\draw[thick] (D4_1)--(D4_2)--(D4_3)--(D4_4)--(D4_5)--(D4_1);
\draw[thick] (D5_1)--(D5_2)--(D5_3)--(D5_4)--(D5_1)--(D5_3)--(D5_2)--(D5_4);
\draw[thick] (D6_1)--(D6_2)--(D6_3)--(D6_4)--(D6_1)--(D6_3)--(D6_2)--(D6_4);

\draw[name path=d1,thick] (D3_1)--(D3_3);
\draw[name path=d2,thick] (D3_2)--(D3_4);
\draw[name path=d3,thick] (D3_3)--(D3_5);
\draw[name path=d4,thick] (D3_4)--(D3_1);
\draw[name path=d5,thick] (D3_5)--(D3_2);
\path[name intersections={of=d1 and d2,by=d6}];
\path[name intersections={of=d2 and d3,by=d7}];
\path[name intersections={of=d3 and d4,by=d8}];
\path[name intersections={of=d4 and d5,by=d9}];
\path[name intersections={of=d5 and d1,by=d10}];
\draw[fill=black] (d6) circle (2pt);
\draw[fill=black] (d7) circle (2pt);
\draw[fill=black] (d8) circle (2pt);
\draw[fill=black] (d9) circle (2pt);
\draw[fill=black] (d10) circle (2pt);
\draw[thick] (D_3)--(d6);
\draw[thick] (D_3)--(d7);
\draw[thick] (D_3)--(d8);
\draw[thick] (D_3)--(d9);
\draw[thick] (D_3)--(d10);

\draw[name path=d11,thick] (D4_1)--(D4_3);
\draw[name path=d12,thick] (D4_2)--(D4_4);
\draw[name path=d13,thick] (D4_3)--(D4_5);
\draw[name path=d14,thick] (D4_4)--(D4_1);
\draw[name path=d15,thick] (D4_5)--(D4_2);
\path[name intersections={of=d11 and d12,by=d16}];
\path[name intersections={of=d12 and d13,by=d17}];
\path[name intersections={of=d13 and d14,by=d18}];
\path[name intersections={of=d14 and d15,by=d19}];
\path[name intersections={of=d15 and d11,by=d20}];
\draw[fill=black] (d16) circle (2pt);
\draw[fill=black] (d17) circle (2pt);
\draw[fill=black] (d18) circle (2pt);
\draw[fill=black] (d19) circle (2pt);
\draw[fill=black] (d20) circle (2pt);
\draw[thick] (D_4)--(d16);
\draw[thick] (D_4)--(d17);
\draw[thick] (D_4)--(d18);
\draw[thick] (D_4)--(d19);
\draw[thick] (D_4)--(d20);

\draw[thick] (D1_1)--++(0,1);
\draw[thick] (D1_2) to [out=120, in=150] (D6_3);
\draw[thick] (D1_3)--(D2_3);
\draw[thick] (D2_1)--++({sqrt(3)/2},0.5);
\draw[thick] (D2_2)--(D3_1);
\draw[thick] (D3_2)--++(-1,0);
\draw[thick] (D3_3)--(D4_2);
\draw[thick] (D3_4)--++({sqrt(3)/2},-0.5);
\draw[thick] (D3_5)--++(1,0); 
\draw[thick] (D4_1)--(D5_4);
\draw[thick] (D4_3)--++({-sqrt(3)/2},-0.5);
\draw[thick] (D4_4)--++(0,-1);
\draw[thick] (D4_5)--++({sqrt(3)/2},-0.5);
\draw[thick] (D5_1)--++({sqrt(2)/2},{sqrt(2)/2});
\draw[thick] (D5_2) to [bend left] (D6_2);
\draw[thick] (D5_3)--++({-sqrt(2)/2},{-sqrt(2)/2});
\draw[thick] (D6_1)--++(1,0);
\draw[thick] (D6_4)--++({sqrt(3)/2},-0.5);

\end{tikzpicture}
 \caption{Canonical inflation $G_{\textrm{infl.}}$ of $G$ and the graphs $G_{\textrm{infl.}}^{+}$ and $(G_{\rho})_{\textrm{infl.}}^{+}$ obtained from $G_{\textrm{infl.}}$ by attaching gadgets and a crossing operation.}
 \label{fig:2}
\end{figure}
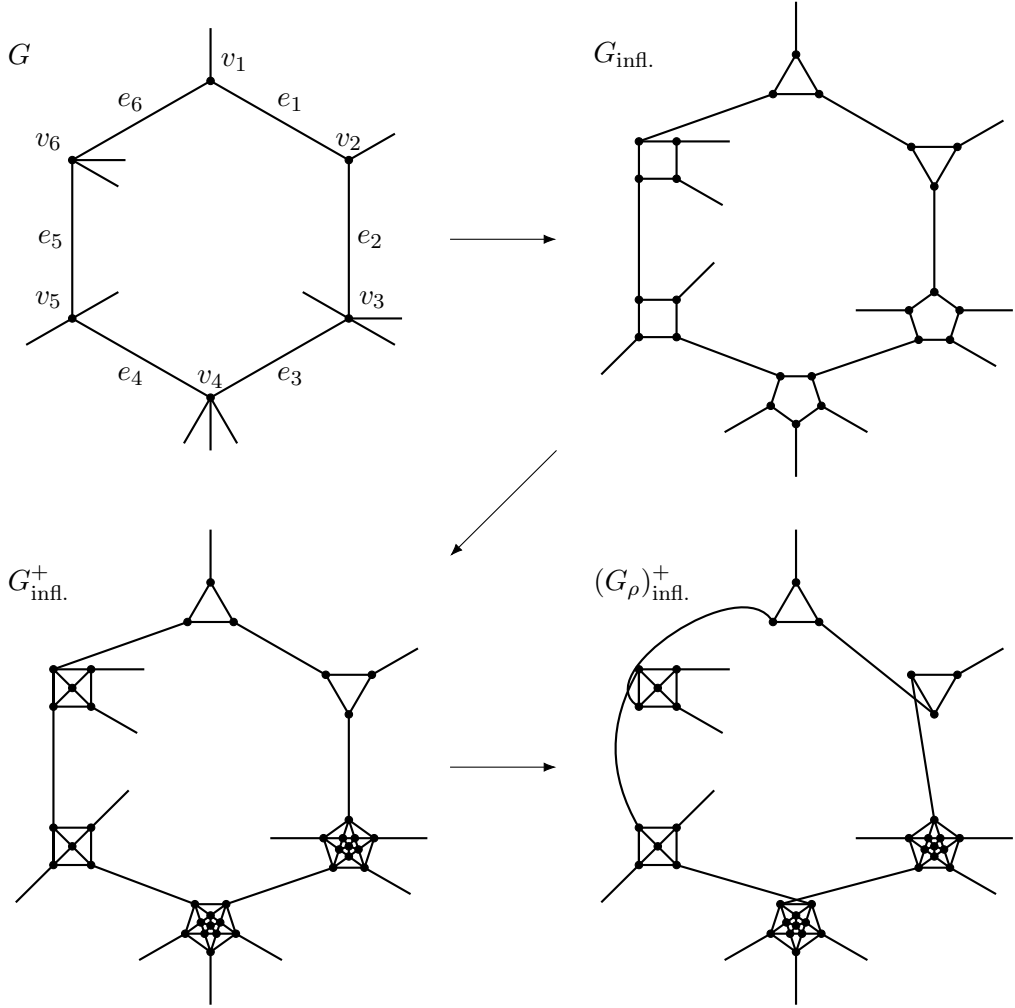

\section{Proofs of the main theorems}
\label{sec:proof}
In this section, we prove Theorems \ref{thm1}--\ref{thm4}. 

\begin{proof}[Proof of Theorem \ref{thm1}]

Let $G$ be a $4$-edge-connected $1$-plane graph. 
Let $H$ be a spanning plane subgraph of $G$ such that the number of components of $H$ is minimum. 
Suppose, for the sake of contradiction, that let $C_1, C_2, \ldots, C_l$ be the components of $H$ where $l \ge 2$. 
For each component $C_i$ of $H$, make $\tilde{C}_i$ by adding as many as edges 
of $G$ with the following conditions: 

\begin{itemize}
\item[(i)]
$V(\tilde{C}_i) = V(C_i)$. 

\item[(ii)]
If $i\ne j$, then any edge of $\tilde{C}_i$ does not cross
with any edge of $\tilde{C}_j$.

\item[(iii)]
If $e$ is a noncut edge of $\tilde{C}_i$, namely, 
$e \in E(\tilde{C}_i)$ with $\tilde{C}_i-e$ being connected, then 
$\tilde{C}_i-e$ contains a connected plane subgraph that spans $V(\tilde{C}_i)$. 
\end{itemize}

Here note that $C_i$ itself satisfies these three conditions. 
Let $\tilde{H}$ be $\bigcup_{1\le i\le l} \tilde{C_i}$. 
Let $A_i \subseteq E(\tilde{C}_i)$ be the set of all cut edges of $\tilde{C}_i$. 

In the following argument, we refer to
any component of $\bigcup_{i=1}^l \tilde{C}_i\setminus A_i$ as a {\em block} of $\tilde{H}$. 

\medskip\noindent
{\bf Claim 1}. For each edge $uv$ of $E(G)\setminus E(\tilde{H})$ that connects different
blocks of $\tilde{H}$, $uv$ crosses with a cut edge of some $\tilde{C}_i$.
\medskip

{\it Proof}.
Let $uv$ be an edge of $E(G)\setminus E(\tilde{H})$ such that $u$ and $v$ are
contained in different blocks $B_u$ and $B_v$, respectively.

\medskip
Case 1: $B_u\subseteq \tilde{C}_i$ and $B_v\subseteq \tilde{C}_j$ with $i\ne j$.

If $uv$ does not cross with any edge of $\tilde{H}$, then we can add the edge $uv$ to $H$
to obtain a plane subgraph of $G$ with $l-1$ components,
which contradicts the minimality of the number of components of $H$.
Suppose that $uv$ crosses with an edge $xy\in E(\tilde{H})$, say $xy\in E(\tilde{C}_k)$,
(possibly $k\in\{i, j\}$).
We may assume $\tilde{C}_k - xy$ is connected.
By the condition~(iii) of $\tilde{C}_k$, we have a plane subgraph $C'_k\subseteq \tilde{C}_k-xy$
which spans $V(\tilde{C}_k)=V(C_k)$.
Then, by replacing $C_k$ with $C'_k$ in $H$, and by adding the edge $uv$,
we obtain a plane subgraph of $G$ consisting of $l-1$ components,
a contradiction.

\medskip
Case 2: $B_u$ and $B_v$ are in a same component, say $\tilde{C}_i$, of $\tilde{H}$.

Suppose first that $uv$ does not cross with any edge of $\tilde{H}$.
Then, we claim that $\tilde{C}'_i=\tilde{C}_i+uv$ satisfies the condition~(iii).
Let $e$ be a noncut edge of $\tilde{C}'_i$.
If $e$ is a noncut edge of $\tilde{C}_i$ or $e=uv$, then
the condition~(iii) for $\tilde{C}_i$ immediately implies the existence of
a desired plane subgraph of $\tilde{C}'_i-e$.
Consider the case where $e$ is a cut edge of $\tilde{C}_i$ but
a noncut edge of $\tilde{C}'_i$.
Note that in this case $u$ and $v$ are in the different components of $\tilde{C}_i-e$.
Then, taking any connected plane subgraph $C'_i$ of $\tilde{C}_i$ with $V(C'_i)=V(\tilde{C}_i)$,
we obtain a desired plane subgraph $C'_i+uv-e$ of $\tilde{C}'_i$.
Thus, $\tilde{C}'_i$ satisfies the condition~(iii),
which contradicts the maximality of $\tilde{C_i}$.

Next suppose that $uv$ crosses with an edge $xy$ of $\tilde{H}$, say $xy\in E(\tilde{C}_k)$.
If $xy$ is a cut edge of $\tilde{C}_k$, then we are done.
Suppose $xy$ is a noncut edge.
If $k\ne i$, then this cannot occur, namely, the edge $xy$ would be a cut edge of $\tilde{C}_k$. 
For otherwise, there exists an $xy$-path of $\tilde{C}_k-xy$,
while $x$ and $y$ are inside and outside of a cycle $D$ of $\tilde{C}'_i$ containing $uv$,
which contradicts the condition~(ii) for $\tilde{C}_i$ and $\tilde{C}_k$.
So we may assume that $k=i$. 
Let $\tilde{C}'_i=\tilde{C}_i+uv$.
We claim that $\tilde{C}'_i$ satisfies the condition~(iii).
If $e$ is a noncut edge of $\tilde{C}_i$ or $e=uv$, then
the condition~(iii) for $\tilde{C}_i$ immediately implies the existence of
a desired plane subgraph of $\tilde{C}'_i-e$.
Consider the case where $e$ is a cut edge of $\tilde{C}_i$ but
a noncut edge of $\tilde{C}'_i$.
Note that in this case $u$ and $v$ are in the different components of $\tilde{C}_i-e$.
Since $xy$ is a noncut edge of $\tilde{C}_i$, we can
take a connected plane subgraph $C'_i$ of $\tilde{C}_i$ with $V(C'_i)=V(\tilde{C}_i)$ avoiding $xy$ by the condition~(iii). 
Then, we obtain a desired connected plane subgraph $C'_i+uv-e$ of $\tilde{C}'_i$.
Thus, $\tilde{C}'_i$ satisfies the condition~(iii),
which contradicts the maximality of $\tilde{C_i}$.
\hfill $\square$
\bigskip

For a block $B$ of $\tilde{H}$, let $d(B)$ denote the number of edges in $G$
connecting $V(B)$ and $V(G)\setminus V(B)$.
Then, the number of edges in $G$ connecting different blocks is
$\frac{1}{2}\sum_{B:{\rm blocks}}d(B)$.
Among those edges, $\sum_{i=1}^l |A_i|$ edges are in $\tilde{H}$.
By Claim 1, the number of edges $e$ in $E(G)\setminus E(\tilde{H})$ 
that connects different blocks of $\tilde{H}$ is
less than the number of cut edges in $\tilde{H}$.
Thus,
\[
\frac{1}{2}\sum_{B:{\rm blocks}} d(B) - \sum_{i=1}^l |A_i| \le \sum_{i=1}^l |A_i|, 
\]
which implies
\[
\sum_{B:{\rm blocks}} d(B) \le 4\sum_{i=1}^l |A_i|.
\]
On the other hand, since $G$ is $4$-edge-connected,
$d(B)\ge 4$ for each block $B$, and the number of blocks in $\tilde{H}$
is $\sum_{i=1}^l (|A_i|+1)$. This implies that
\[
\sum_{B:{\rm blocks}} d(B) \ge 4\sum_{i=1}^l (|A_i|+1),
\]
which contradicts the previous inequality.
\hfill
\end{proof}

\begin{proof}[Proof of Theorem \ref{thm2}]
Let $G_{2k}$ ($k\geq 2$) denote a prism of an even cycle $C_{2k}$ (the cartesian product of $K_2$ and $C_{2k}$) embedded on the plane with the outer cycle $u_1 \cdots u_{2k}$ and the inner cycle $v_1 \cdots v_{2k}$ with $2k$ edges $u_iv_i$; that is, $u_iu_{i+1}v_{i+1}v_i$ for $i\in \{1, \ldots, 2k-1\}$ and $u_{2k}u_1v_1v_{2k}$ bound quadrangular faces of $G_{2k}$. 
Since $G_{2k}$ is $3$-edge connected, $(G_{2k})_{\textrm{infl.}}^{+}$ is a $3$-connected plane graph by Propositions~\ref{prop_1-planarity}~and~\ref{prop_k-conn}. 
Next let $((G_{2k})_{\rho})_{\textrm{infl.}}^{+}$ denote the $1$-plane graph obtained from $(G_{2k})_{\textrm{infl.}}^{+}$ 
by a crossing operation along the even cycle $v_1 \cdots v_{2k}$ and the paths $v_iu_iu_{i+1}$ for $i \in \{1, \ldots, 2k\}$, where $u_{2k+1} = u_1$, in $G_{2k}$; as we mentioned, there exists a suitable $\rho$. 
Note that all edges corresponding to original edges of $G_{2k}$ become crossing edges in 
$((G_{2k})_{\rho})_{\textrm{infl.}}^{+}$; see Figure~\ref{fig:4} for example for $k=2$. 
Since the number of crossing edges of $((G_{2k})_{\rho})_{\textrm{infl.}}^{+}$ is $6k$, 
any spanning plane subgraph $F$ of $((G_{2k})_{\rho})_{\textrm{infl.}}^{+}$ 
contains at most $3k$ edges corresponding to edges of $G_{2k}$. 
This implies that the number of edges of $F$ which join different inflated $3$-cycles is at most $3k < 4k-1 = |V(G_{2k})|-1$, that is, $F$ is disconnected. 
\end{proof}

\noindent
\textbf{Remark:} In the proof of Theorem~\ref{thm2}, one can take an arbitrary 3-edge connected cubic plane graph of order at least $8$ instead of $G_{2k}$; for every connected graph $G$, the edge set $E(G)$ can be partitioned into paths of length 2 (and one edge when $|E(G)|$ is odd), see \cite{Bo}.

\begin{figure}[th]
 \centering
\begin{tikzpicture}[scale=0.5]
\draw[fill=black] (1,0) circle (4pt);
\draw[fill=black] (6,0) circle (4pt);
\draw[fill=black] (0,1) circle (4pt);
\draw[fill=black] (1,1) circle (4pt);
\draw[fill=black] (6,1) circle (4pt);
\draw[fill=black] (7,1) circle (4pt);
\draw[fill=black] (2,2) circle (4pt);
\draw[fill=black] (3,2) circle (4pt);
\draw[fill=black] (4,2) circle (4pt);
\draw[fill=black] (5,2) circle (4pt);
\draw[fill=black] (2,3) circle (4pt);
\draw[fill=black] (5,3) circle (4pt);
\draw[fill=black] (2,4) circle (4pt);
\draw[fill=black] (5,4) circle (4pt);
\draw[fill=black] (2,5) circle (4pt);
\draw[fill=black] (3,5) circle (4pt);
\draw[fill=black] (4,5) circle (4pt);
\draw[fill=black] (5,5) circle (4pt);
\draw[fill=black] (0,6) circle (4pt);
\draw[fill=black] (1,6) circle (4pt);
\draw[fill=black] (6,6) circle (4pt);
\draw[fill=black] (7,6) circle (4pt);
\draw[fill=black] (1,7) circle (4pt);
\draw[fill=black] (6,7) circle (4pt);

\draw[thick] (1,0)--(1,1)--(0,1)--(1,0)--(6,1);
\draw[thick] (7,1)--(6,1)--(6,0)--(7,1)--(6,6);
\draw[thick] (6,7)--(6,6)--(7,6)--(6,7)--(1,6);
\draw[thick] (0,6)--(1,6)--(1,7)--(0,6)--(1,1);
\draw[thick] (3,2)--(2,3)--(2,2)--(3,2)--(5,3);
\draw[thick] (4,2)--(5,2)--(5,3)--(4,2)--(5,4);
\draw[thick] (4,5)--(5,4)--(5,5)--(4,5)--(2,4);
\draw[thick] (3,5)--(2,5)--(2,4)--(3,5)--(2,3);
\draw[thick] (0,1)--(2,2);
\draw[thick] (6,0)--(5,2);
\draw[thick] (5,5)--(7,6);
\draw[thick] (1,7)--(2,5);
\end{tikzpicture}
 \caption{1-plane graph $((G_{4})_{\rho})_{\textrm{infl.}}^{+}$ in the proof of Theorem \ref{thm2}.}
 \label{fig:4}
\end{figure}
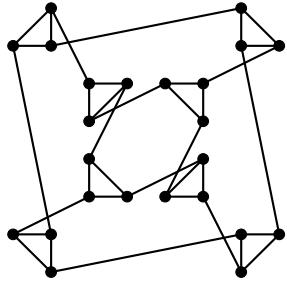

\begin{proof}[Proof of Theorem \ref{thm3}]
Let $G$ be a $3$-edge connected cubic plane graph without hamiltonian cycles. 
In fact, there are infinitely many such graphs; see \cite{HM, Tu2} for example. 
It is shown by Tait's theorem~\cite{Ta} and the \emph{Four Color Theorem}~\cite{AH} that there exists a proper $3$-edge-coloring $c: E(G) \to \{1, 2, 3\}$ of $G$. 
Let $M$ be a perfect matching of $G$ corresponding to the edges colored by $1$, and replace each edge of $M$ with parallel edges; 
call the resulting $4$-regular plane multigraph $G_M$. 
The graph $G_M$ is $4$-edge connected since every $3$-edge cut in $G$ consists of the edges colored by 1, 2, and 3, respectively (by the \emph{Parity Lemma}; see \cite{Is} for example).
Next, for $(G_{M})_{\textrm{infl.}}^{+}$ of $G_M$, 
consider $((G_{M})_{\rho})_{\textrm{infl.}}^{+}$ that is 
a $1$-plane graph obtained from $(G_{M})_{\textrm{infl.}}^{+}$ 
by a crossing operation along 
\begin{itemize}
\item
every pair of parallel edges in $G_M$, and 
\item
all cycles in the $2$-factor (i.e., spanning 2-regular subgraph) $E(G)-M$ of $G_M$, which have even length because the edges are colored $2$ and $3$. 
\end{itemize}

See Figure~\ref{fig:1} for example, however, note that $G$ has a hamiltonian cycle in this example for simplicity. 
Then $((G_{M})_{\rho})_{\textrm{infl.}}^{+}$ is a $4$-connected $4$-regular $1$-plane graph by Propositions~\ref{prop_1-planarity}~and~\ref{prop_k-conn}. 
Let $F$ be a spanning plane subgraph of $((G_{M})_{\rho})_{\textrm{infl.}}^{+}$. 
Suppose, for sake of contradiction, that $F$ is $2$-connected. 
For every $4$-gadget corresponding to a vertex $v$ in $G_M$, at least $2$ edges of $e_1, e_2, e_3, e_4$, which are incident to $v$ in $G_M$, must be contained in $F$; 
in fact, exactly $2$ edges can be contained in $F$ by construction. 
Then, $F$ corresponds to a $2$-factor of $G$. 
Since $G$ does not have a hamiltonian cycle, $F$ cannot be connected, which is a contradiction. 
\end{proof}

\begin{figure}[ht]
\centering
\begin{tikzpicture}[scale=0.7]
\def\y{9}
\coordinate [label=right: {$G$}] (fig:1) at (-3,2.8);
\coordinate [label=right: {$G_M$}] (fig:2) at (\y-3,2.8);
\coordinate [label=right: {$(G_{M})_{\textrm{infl.}}^{+}$}] (fig:3) at (-3.5,-4);
\coordinate [label=right: {$((G_{M})_{\rho})_{\textrm{infl.}}^{+}$}] (fig:4) at (\y-4,-4);
\coordinate [label=left: {$1$}] (1_1) at (0,1.5);
\coordinate [label=above: {$1$}] (1_2) at (0,-0.5);
\coordinate [label=above: {$1$}] (1_3) at (0,-1.5);
\coordinate [label=center: {$2$}] (2_1) at (-1.6,1);
\coordinate [label=center: {$2$}] (2_2) at (-0.6,0.5);
\coordinate [label=center: {$2$}] (2_3) at (1.6,-0.6);
\coordinate [label=center: {$3$}] (3_1) at (1.6,1);
\coordinate [label=center: {$3$}] (3_2) at (0.6,0.5);
\coordinate [label=center: {$3$}] (3_3) at (-1.6,-0.6);

\draw [->, >=Latex] (3.5,0) -- (5.5,0);
\draw [->, >=Latex] (5.5,-2.5) -- (3.5,-4.5);
\draw [->, >=Latex] (3.5,-7) -- (5.5,-7);

\foreach \k in{1,...,3}
   {
   \coordinate (A_\k) at ({3*cos((4*\k-1)*pi/6 r)},{3*sin((4*\k-1)*pi/6 r)});
   \coordinate (B_\k) at ({cos((4*\k-1)*pi/6 r)},{sin((4*\k-1)*pi/6 r)});
   \draw[fill=black] (A_\k) circle (2pt);
   \draw[fill=black] (B_\k) circle (2pt);
   }

\draw[thick] (A_1)--(A_2)--(A_3)--(A_1)--(B_1)--(B_2)--(B_3)--(B_1);
\draw[thick] (A_2)--(B_2);
\draw[thick] (A_3)--(B_3);

\foreach \k in{1,...,3}
   {
   \coordinate (C_\k) at ({\y+3*cos((4*\k-1)*pi/6 r)},{3*sin((4*\k-1)*pi/6 r)});
   \coordinate (D_\k) at ({\y+cos((4*\k-1)*pi/6 r)},{sin((4*\k-1)*pi/6 r)});
   \draw[fill=black] (C_\k) circle (2pt);
   \draw[fill=black] (D_\k) circle (2pt);
   }

\draw[thick] (C_1)--(C_2)--(D_2)--(D_1)--(D_3)--(C_3)--(C_1);
\draw[thick] (C_1) to [out=290, in=70] (D_1);
\draw[thick] (C_1) to [out=250, in=110] (D_1);
\draw[thick] (D_2) to [bend left] (D_3);
\draw[thick] (D_2) to [bend right] (D_3);
\draw[thick] (C_2) to [out=10, in=170] (C_3);
\draw[thick] (C_2) to [out=350, in=190] (C_3);

\def\x{0.3}
\foreach \k in{1,...,3}
   {
   \coordinate (E_\k) at ({3*cos((4*\k-1)*pi/6 r)},{-7+3*sin((4*\k-1)*pi/6 r)});
   \coordinate (E1_\k) at ({\x+3*cos((4*\k-1)*pi/6 r)},{-7+\x+3*sin((4*\k-1)*pi/6 r)});
   \coordinate (E2_\k) at ({-\x+3*cos((4*\k-1)*pi/6 r)},{-7+\x+3*sin((4*\k-1)*pi/6 r)});
   \coordinate (E3_\k) at ({-\x+3*cos((4*\k-1)*pi/6 r)},{-7-\x+3*sin((4*\k-1)*pi/6 r)});
   \coordinate (E4_\k) at ({\x+3*cos((4*\k-1)*pi/6 r)},{-7-\x+3*sin((4*\k-1)*pi/6 r)});
   \coordinate (F_\k) at ({cos((4*\k-1)*pi/6 r)},{-7+sin((4*\k-1)*pi/6 r)});
   \coordinate (F1_\k) at ({\x+cos((4*\k-1)*pi/6 r)},{-7+\x+sin((4*\k-1)*pi/6 r)});
   \coordinate (F2_\k) at ({-\x+cos((4*\k-1)*pi/6 r)},{-7+\x+sin((4*\k-1)*pi/6 r)});
   \coordinate (F3_\k) at ({-\x+cos((4*\k-1)*pi/6 r)},{-7-\x+sin((4*\k-1)*pi/6 r)});
   \coordinate (F4_\k) at ({\x+cos((4*\k-1)*pi/6 r)},{-7-\x+sin((4*\k-1)*pi/6 r)});
   \draw[fill=black] (E_\k) circle (2pt);
   \draw[fill=black] (E1_\k) circle (2pt);
   \draw[fill=black] (E2_\k) circle (2pt);
   \draw[fill=black] (E3_\k) circle (2pt);
   \draw[fill=black] (E4_\k) circle (2pt);
   \draw[fill=black] (F_\k) circle (2pt);
   \draw[fill=black] (F1_\k) circle (2pt);
   \draw[fill=black] (F2_\k) circle (2pt);
   \draw[fill=black] (F3_\k) circle (2pt);
   \draw[fill=black] (F4_\k) circle (2pt);
   \draw[thick] (E1_\k)--(E2_\k)--(E3_\k)--(E4_\k)--(E1_\k)--(E3_\k)--(E2_\k)--(E4_\k);
   \draw[thick] (F1_\k)--(F2_\k)--(F3_\k)--(F4_\k)--(F1_\k)--(F3_\k)--(F2_\k)--(F4_\k);
   }
\draw[thick] (E1_1)--(E1_3);
\draw[thick] (E2_1)--(E2_2);
\draw[thick] (E3_1)--(F2_1);
\draw[thick] (E4_1)--(F1_1);
\draw[thick] (F3_1)--(F2_2);
\draw[thick] (F4_1)--(F1_3);
\draw[thick] (F1_2)--(F2_3);
\draw[thick] (F4_2)--(F3_3);
\draw[thick] (E1_2)--(F3_2);
\draw[thick] (E2_3)--(F4_3);
\draw[thick] (E4_2)--(E3_3);
\draw[thick] (E3_2) to [out=345, in=195] (E4_3);

\draw[thick] (E1_1)--(E1_3);
\draw[thick] (E2_1)--(E2_2);
\draw[thick] (E3_1)--(F2_1);
\draw[thick] (E4_1)--(F1_1);
\draw[thick] (F3_1)--(F2_2);
\draw[thick] (F4_1)--(F1_3);
\draw[thick] (F1_2)--(F2_3);
\draw[thick] (F4_2)--(F3_3);
\draw[thick] (E1_2)--(F3_2);
\draw[thick] (E2_3)--(F4_3);
\draw[thick] (E4_2)--(E3_3);
\draw[thick] (E3_2) to [out=345, in=195] (E4_3);

\foreach \k in{1,...,3}
   {
   \coordinate (E'_\k) at ({\y+3*cos((4*\k-1)*pi/6 r)},{-7+3*sin((4*\k-1)*pi/6 r)});
   \coordinate (E1'_\k) at ({\y+\x+3*cos((4*\k-1)*pi/6 r)},{-7+\x+3*sin((4*\k-1)*pi/6 r)});
   \coordinate (E2'_\k) at ({\y-\x+3*cos((4*\k-1)*pi/6 r)},{-7+\x+3*sin((4*\k-1)*pi/6 r)});
   \coordinate (E3'_\k) at ({\y-\x+3*cos((4*\k-1)*pi/6 r)},{-7-\x+3*sin((4*\k-1)*pi/6 r)});
   \coordinate (E4'_\k) at ({\y+\x+3*cos((4*\k-1)*pi/6 r)},{-7-\x+3*sin((4*\k-1)*pi/6 r)});
   \coordinate (F'_\k) at ({\y+cos((4*\k-1)*pi/6 r)},{-7+sin((4*\k-1)*pi/6 r)});
   \coordinate (F1'_\k) at ({\y+\x+cos((4*\k-1)*pi/6 r)},{-7+\x+sin((4*\k-1)*pi/6 r)});
   \coordinate (F2'_\k) at ({\y-\x+cos((4*\k-1)*pi/6 r)},{-7+\x+sin((4*\k-1)*pi/6 r)});
   \coordinate (F3'_\k) at ({\y-\x+cos((4*\k-1)*pi/6 r)},{-7-\x+sin((4*\k-1)*pi/6 r)});
   \coordinate (F4'_\k) at ({\y+\x+cos((4*\k-1)*pi/6 r)},{-7-\x+sin((4*\k-1)*pi/6 r)});
   \draw[fill=black] (E'_\k) circle (2pt);
   \draw[fill=black] (E1'_\k) circle (2pt);
   \draw[fill=black] (E2'_\k) circle (2pt);
   \draw[fill=black] (E3'_\k) circle (2pt);
   \draw[fill=black] (E4'_\k) circle (2pt);
   \draw[fill=black] (F'_\k) circle (2pt);
   \draw[fill=black] (F1'_\k) circle (2pt);
   \draw[fill=black] (F2'_\k) circle (2pt);
   \draw[fill=black] (F3'_\k) circle (2pt);
   \draw[fill=black] (F4'_\k) circle (2pt);
   \draw[thick] (E1'_\k)--(E2'_\k)--(E3'_\k)--(E4'_\k)--(E1'_\k)--(E3'_\k)--(E2'_\k)--(E4'_\k);
   \draw[thick] (F1'_\k)--(F2'_\k)--(F3'_\k)--(F4'_\k)--(F1'_\k)--(F3'_\k)--(F2'_\k)--(F4'_\k);
   }

\draw[thick] (E1'_1)--(E2'_3);
\draw[thick] (E2'_1)--(E1'_2);
\draw[thick] (E3'_1)--(F1'_1);
\draw[thick] (E4'_1)--(F2'_1);
\draw[thick] (F3'_1)--(F1'_3);
\draw[thick] (F4'_1)--(F2'_2);
\draw[thick] (F1'_2)--(F3'_3);
\draw[thick] (F4'_2)--(F2'_3);
\draw[thick] (E2'_2)--(F3'_2);
\draw[thick] (E1'_3)--(F4'_3);
\draw[thick] (E4'_2) to [out=345, in=195] (E4'_3);
\draw[thick] (E3'_2) to [out=345, in=195] (E3'_3);
\end{tikzpicture}
 \caption{$4$-connected $4$-regular 1-plane graph $((G_{M})_{\rho})_{\textrm{infl.}}^{+}$ obtained from a $3$-edge connected cubic plane graph $G$.}
 \label{fig:1}
\end{figure}
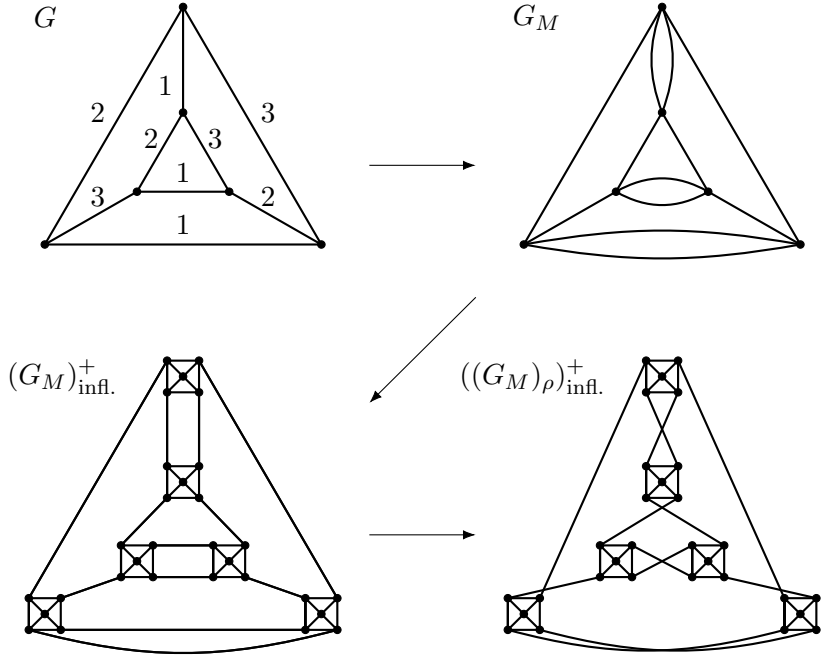

\begin{proof}[Proof of Theorem \ref{thm4}]
Let $k \ge 2$ be an integer. 
Let $C_{2k} = v_1e_1v_2e_2 \cdots v_{2k}e_{2k}v_1$ be an even cycle with a plane embedding. 
Replace each edge $e_i$ with three parallel edges $e_{i1}, e_{i2}, e_{i3}$ for $i\in \{1 \le i \le 2k\}$ so that $e_{11}, \ldots, e_{(2k)1}$ bounds an inner face; call the resulting $6$-edge-connected $6$-regular plane multigraph $I_{2k}$. 
Take the canonical inflation of $I_{2k}$ and attach an $6$-gadget to each inflated $6$-cycle to obtain $(I_{2k})_{\textrm{infl.}}^{+}$. 
Now, consider a $1$-plane graph 
$((I_{2k})_{\rho})_{\textrm{infl.}}^{+}$ obtained from $(I_{2k})_{\textrm{infl.}}^{+}$ by a crossing operation along
\begin{itemize}
\item
every pair of parallel edges $e_{i2}$ and $e_{i3}$ in $I_{2k}$ ($i \in \{1, \ldots, 2k\}$), and 
\item
the $2k$-cycle $v_1e_1v_2e_2\cdots v_{2k}e_{2k}v_1$ in $I_{2k}$. 
\end{itemize}

See Figure~\ref{fig:3} for example for $k=3$. 
The resulting graph $((I_{2k})_{\rho})_{\textrm{infl.}}^{+}$ is a $6$-connected $6$-regular $1$-plane graph by Propositions~\ref{prop_1-planarity}~and~\ref{prop_k-conn}. 
Let $F$ be a spanning plane subgraph of $((I_{2k})_{\rho})_{\textrm{infl.}}^{+}$. 
On the $6$-gadget corresponding to $v_2$ in $I_{2k}$, at most $3$ edges of $e_{11}, e_{12}, e_{13}, e_{21}, e_{22}$, $e_{23}$, which are incident to $v_2$ in $I_{2k}$, is contained in $F$ by construction. 
So is the $6$-gadget corresponding to $v_4$, too. 
Thus, there exists a $2$-edge cut in $F$, which corresponds to $\{e_s, e_t\}$ in $C_{2k}$ for some $s\in\{1,2\}$ and some $t\in\{3,4\}$, and $F$ is not $3$-connected. 
\end{proof}

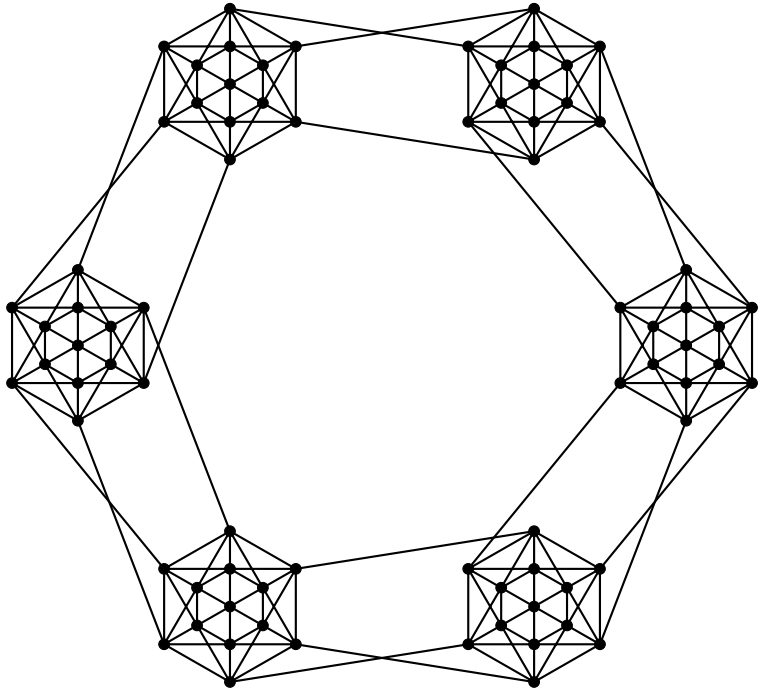
\begin{figure}[th]
\centering
\begin{tikzpicture}
\foreach \k in{0,...,5}
   {
\coordinate (a_\k) at ({4*cos(\k*pi/3 r)+0},{4*sin(\k*pi/3 r)+0});
\coordinate (b_\k) at ({4*cos(\k*pi/3 r)+0},{4*sin(\k*pi/3 r)+1});
\coordinate (c_\k) at ({4*cos(\k*pi/3 r)+cos(5*pi/6 r)},{4*sin(\k*pi/3 r)+sin(5*pi/6 r)});
\coordinate (d_\k) at ({4*cos(\k*pi/3 r)+cos(7*pi/6 r)},{4*sin(\k*pi/3 r)+sin(7*pi/6 r)});
\coordinate (e_\k) at ({4*cos(\k*pi/3 r)+0},{4*sin(\k*pi/3 r)-1});
\coordinate (f_\k) at ({4*cos(\k*pi/3 r)+cos(11*pi/6 r)},{4*sin(\k*pi/3 r)+sin(11*pi/6 r)});
\coordinate (g_\k) at ({4*cos(\k*pi/3 r)+cos(13*pi/6 r)},{4*sin(\k*pi/3 r)+sin(13*pi/6 r)});
\coordinate (b'_\k) at ({4*cos(\k*pi/3 r)+0},{4*sin(\k*pi/3 r)+0.5});
\coordinate (c'_\k) at ({4*cos(\k*pi/3 r)+0.5*cos(5*pi/6 r)},{4*sin(\k*pi/3 r)+0.5*sin(5*pi/6 r)});
\coordinate (d'_\k) at ({4*cos(\k*pi/3 r)+0.5*cos(7*pi/6 r)},{4*sin(\k*pi/3 r)+0.5*sin(7*pi/6 r)});
\coordinate (e'_\k) at ({4*cos(\k*pi/3 r)+0},{4*sin(\k*pi/3 r)-0.5});
\coordinate (f'_\k) at ({4*cos(\k*pi/3 r)+0.5*cos(11*pi/6 r)},{4*sin(\k*pi/3 r)+0.5*sin(11*pi/6 r)});
\coordinate (g'_\k) at ({4*cos(\k*pi/3 r)+0.5*cos(13*pi/6 r)},{4*sin(\k*pi/3 r)+0.5*sin(13*pi/6 r)});
\draw[fill=black] (a_\k) circle (2pt);
\draw[fill=black] (b_\k) circle (2pt);
\draw[fill=black] (c_\k) circle (2pt);
\draw[fill=black] (d_\k) circle (2pt);
\draw[fill=black] (e_\k) circle (2pt);
\draw[fill=black] (f_\k) circle (2pt);
\draw[fill=black] (g_\k) circle (2pt);
\draw[fill=black] (b'_\k) circle (2pt);
\draw[fill=black] (c'_\k) circle (2pt);
\draw[fill=black] (d'_\k) circle (2pt);
\draw[fill=black] (e'_\k) circle (2pt);
\draw[fill=black] (f'_\k) circle (2pt);
\draw[fill=black] (g'_\k) circle (2pt);
\draw[thick] (b_\k)--(c_\k)--(d_\k)--(e_\k)--(f_\k)--(g_\k)--(b_\k);
\draw[thick] (b_\k)--(d_\k)--(f_\k)--(b_\k)--(e_\k)--(g_\k)--(c_\k)--(e_\k);
\draw[thick] (c_\k)--(f_\k);
\draw[thick] (d_\k)--(g_\k);
\draw[thick] (b'_\k)--(c'_\k)--(d'_\k)--(e'_\k)--(f'_\k)--(g'_\k)--(b'_\k);
   }

\draw[thick] (b_0)--(g_1);
\draw[thick] (g_0)--(f_1);
\draw[thick] (b_1)--(g_2);
\draw[thick] (c_1)--(b_2);
\draw[thick] (c_2)--(b_3);
\draw[thick] (d_2)--(c_3);
\draw[thick] (d_3)--(c_4);
\draw[thick] (e_3)--(d_4);
\draw[thick] (e_4)--(d_5);
\draw[thick] (f_4)--(e_5);
\draw[thick] (f_5)--(e_0);
\draw[thick] (g_5)--(f_0);

\draw[thick] (c_0)--(d_1);
\draw[thick] (e_1)--(f_2);
\draw[thick] (e_2)--(f_3);
\draw[thick] (g_3)--(b_4);
\draw[thick] (g_4)--(b_5);
\draw[thick] (c_5)--(d_0);
\end{tikzpicture}
 \caption{$6$-connected $6$-regular $1$-plane graph $((I_{6})_{\rho})_{\textrm{infl.}}^{+}$ in the proof of Theorem~\ref{thm4}.}
 \label{fig:3}
\end{figure}

\section{Concluding remarks}
We consider the condition of $k$ and $l$ such that every $k$-connected 1-plane graph has an $l$-connected spanning plane subgraph; see Table~\ref{table:1}. 
For the upper bound of $l$ for $k=7$, one can easily construct a $7$-connected 1-plane graph $H$ that has no $4$-connected spanning plane subgraphs as follows. 
For example, such a graph $H$ is obtained from a $7$-edge connected $7$-regular plane multigraph $G$ by applying the canonical inflation, attaching $7$-gadgets, and a suitable crossing operation (e.g. $G$ is obtained from $G_{2k}$ in the proof of Theorem~\ref{thm2} by replacing $u_iu_{i+1}$ with parallel edges, $v_iv_{i+1}$ with parallel edges, and $u_iv_i$ with three parallel edges, respectively, for $i\in\{1, \ldots, 2k\}$). 
\begin{table}[th]
 \caption{Condition of $k$ and $l$}
 \label{table:1}
 \centering
  \begin{tabular}{cc}
   \hline
   $k$ & $l$ \\
   \hline
   1 & 0 \\
   2 & 0 \\
   3 & 0 (Theorem~\ref{thm2}) \\
   4 & 1 (Theorems~\ref{thm1}~and~\ref{thm3}) \\
   5 & 1 or 2 \\
   6 & 1 or 2 (Theorem~\ref{thm4})\\
   7 & 1, 2, or 3 \\
   \hline
  \end{tabular}
\end{table}
Lastly, we leave the following problem.

\begin{prob}
Determine the exact value of $l$ in Table \ref{table:1} when $k \in \{5, 6, 7\}$.
\end{prob}

\section*{Acknowledgements}
Noguchi's work is partially supported by JSPS KAKENHI Grant Number JP21K13831. 
Ota's work is partially supported by JSPS KAKENHI Grant Number JP22K03404. 
Suzuki's work is partially supported by JSPS KAKENHI Grant Number JP23K03196.

\end{document}